\documentclass[11pt]{amsart}
\usepackage{amsmath,amssymb,pstricks,mathrsfs}
\usepackage[colorlinks=true,linkcolor=black,citecolor=black,urlcolor=black]{hyperref}

\psset{unit=1pt}
\psset{arrowsize=4pt 1}
\psset{linewidth=.5pt}

\newcommand{\define}{\textbf}

\newcommand{\excise}[1]{}

\newcommand{\isom}{\cong}
\renewcommand{\setminus}{\smallsetminus}
\renewcommand{\phi}{\varphi}

\renewcommand{\tilde}{\widetilde}

\renewcommand{\bar}{\overline}

\newcommand{\BB}{\mathbb{B}}
\newcommand{\CC}{\mathbb{C}}
\newcommand{\EE}{\mathbb{E}}

\newcommand{\NN}{\mathbb{N}}
\newcommand{\PP}{\mathbb{P}}

\newcommand{\ZZ}{\mathbb{Z}}

\newcommand{\qq}{\mathbf{q}}   
\newcommand{\dd}{\mathbf{d}}   

\newcommand{\pt}{\mathrm{pt}}  

\newcommand{\Mbar}{\bar{M}}     
\newcommand{\bdy}{\partial\Mbar}

\newcommand{\qtp}{\circ}        

\newcommand{\ev}{\mathrm{ev}}   

\newcommand{\Bb}{\mathcal{B}}
\newcommand{\Gg}{\mathcal{G}}
\newcommand{\Mm}{\mathcal{M}}
\newcommand{\bbdy}{\partial\Mmbar}
\newcommand{\Mmbar}{\bar{\mathcal{M}}}
\newcommand{\Oo}{\mathcal{O}}  
\newcommand{\Uu}{\mathcal{U}}
\newcommand{\Ww}{\mathcal{W}}
\newcommand{\Xx}{\mathcal{X}}
\newcommand{\Yy}{\mathcal{Y}}
\newcommand{\Zz}{\mathcal{Z}}

\DeclareMathOperator{\codim}{codim}

\DeclareMathOperator{\Hom}{Hom}

\newtheorem{theorem}{Theorem}[section]
\newtheorem{lemma}[theorem]{Lemma}
\newtheorem{proposition}[theorem]{Proposition}

\theoremstyle{definition}
\newtheorem{definition}[theorem]{Definition}
\newtheorem{remark}[theorem]{Remark}

\begin{document}

\title{Positivity of equivariant Gromov-Witten invariants}
\author{Dave Anderson}
\address{Department of Mathematics\\University of Washington\\Seattle, WA 98195}
\email{dandersn@math.washington.edu}
\author{Linda Chen}
\address{Department of Mathematics and Statistics\\Swarthmore College\\Swarthmore, PA 19081}
\email{lchen@swarthmore.edu}
\date{December 3, 2011}
\thanks{DA was partially supported by NSF Grant DMS-0902967.  LC was partially supported by NSF Grant DMS-0908091 and NSF Grant DMS-1101625.}

\begin{abstract}
We show that the equivariant Gromov-Witten invariants of a projective homogeneous space $G/P$ exhibit Graham-positivity: when expressed as polynomials in the positive roots, they have nonnegative coefficients.
\end{abstract}

\maketitle

\section{Introduction}\label{s:intro}

Let $X=G/P$ be a projective homogeneous variety, for a complex reductive Lie group $G$ and parabolic subgroup $P$.  Fix a maximal torus and Borel subgroup $T \subset B \subseteq P$, and let $\Delta = \{\alpha_1,\ldots,\alpha_n\}$ be the corresponding set of simple roots, making the roots of $B$ positive.  Let $W_P\subseteq W$ be the Weyl groups for $P$ and $G$, respectively.  Let $B^-$ be the opposite Borel subgroup.  The classes of the \emph{Schubert varieties} $X(w) = \bar{BwP/P}$ and \emph{opposite Schubert varieties} $Y(w) = \bar{B^-wP/P}$ give Poincar\'e dual bases of the equivariant cohomology ring $H_T^*X$, as $w$ ranges over the set $W^P$ of minimal coset representatives for $W/W_P$.
%
%
Write $x(w)=[X(w)]^T$ and $y(w)=[Y(w)]^T$ for these classes.

A positivity property for multiplication in these bases was proved by Graham:
\begin{theorem}[{\cite{graham}}]\label{t:graham}
Writing
\[
  y(u)\cdot y(v) = \sum_w c_{u,v}^w\, y(w)
\]
in $H_T^*X$, the coefficient $c_{u,v}^w$ lies in $\NN[\alpha_1,\ldots,\alpha_n]$.
\end{theorem}

Following \cite{kim-eq}, the \emph{equivariant Gromov-Witten invariants} are defined as follows.  Let $\dd \in H_2(X,\ZZ)$ be an effective class; taking the basis of Schubert curves $x(s_\alpha)$, one can identify $\dd$ with a tuple of nonnegative integers $(d_1,\ldots,d_k)$.  Let $\Mbar = \Mbar_{0,r+1}(X,\dd)$ denote the Kontsevich moduli space of stable maps.  This comes with $r+1$ \emph{evaluation maps} $\ev_i:\Mbar \to X$, as well as the standard map $\pi\colon \Mbar \to \pt$.

\begin{definition}
The \define{equivariant Gromov-Witten invariant} associated to classes $\sigma_1,\ldots,\sigma_{r+1}$ is
\[
  I^T_\dd(\sigma_1\cdots\sigma_{r+1}) := \pi^T_*( \ev_1^*\sigma_1 \cdots \ev_{r+1}^*\sigma_{r+1} )
\]
in $H_T^*(\pt)$, where $\pi^T_*$ is the equivariant pushforward $H_T^*\Mbar \to H_T^*(\pt)$.
\end{definition}

When $r=2$, these define \emph{equivariant quantum Littlewood-Richardson (EQLR) coefficients}:
\[
  c_{u,v}^{w,\dd} = I^T_\dd( y(u)\cdot y(v) \cdot x(w) ).
\]
The EQLR coefficients were shown to be Graham-positive, in the sense of Theorem~\ref{t:graham}, by Mihalcea in \cite{mihalcea-positivity}.  Remarkably, they define an associative product in the \emph{equivariant (small) quantum cohomology ring} $QH_T^*X$, via
\[
  y(u)\qtp y(v) = \sum_{w,\dd} \qq^\dd\, c_{u,v}^{w,\dd}\, y(w),
\]
so Mihalcea's result is a generalization of Graham's to the setting of equivariant quantum Schubert calculus.

In this note, we will show that the multiple-point equivariant Gromov-Witten invariants are Graham-positive:
\begin{theorem}\label{t:main}
For any elements $v_1,\ldots,v_r,w \in W^P$, the equivariant Gromov-Witten invariant
\[
  I^T_\dd(y(v_1)\cdots y(v_r) \cdot x(w)) 
\]
lies in $\NN[\alpha_1,\ldots,\alpha_n]$.
\end{theorem}


Associativity of the equivariant quantum ring $QH_T^*X$ defines (generalized) EQLR coefficients $c_{v_1,\ldots,v_r}^{w,\dd}$:
\[
  y(v_1)\qtp\cdots\qtp y(v_r) = \sum_{w,\dd} \qq^\dd\, c_{v_1,\ldots,v_r}^{w,\dd}\, y(w).
\]
By induction using the $r=2$ case of Theorem~\ref{t:main}, it follows that these EQLR coefficients are also Graham-positive; indeed, the associativity relations are subtraction-free.  This gives a new proof of Mihalcea's positivity theorem.  For $r>2$, however, the EQLR coefficients $c_{v_1,\ldots,v_r}^{w,\dd}$ are not the same as the equivariant Gromov-Witten invariants 
in Theorem \ref{t:main}.

The proof of Theorem~\ref{t:main} is given in \S\ref{s:transverse}; the idea is to represent the coefficients of this polynomial as degrees of effective zero-cycles, using a transversality argument (Theorem~\ref{t:count}).  An inspection of Mihalcea's proof of positivity for EQLR coefficients suggests that his method should also work for Gromov-Witten invariants, but we find our geometric interpretation of the coefficients appealing.  Moreover, we use the dimension estimates from \S\ref{s:transverse} to derive a Giambelli formula for $QH_T^*(SL_n/P)$ in \cite{ac}.

\begin{remark}
As in \cite{graham}, there is a corresponding positivity theorem with the roles of positive and negative roots interchanged: the Gromov-Witten invariants $I^T_\dd(x(v_1)\cdots x(v_r)\cdot y(w))$ lie in $\NN[-\alpha_1,\ldots,-\alpha_n]$.  All the arguments proceed in exactly the same manner.  In fact, it is this version (for $r=2$) that is treated in \cite{mihalcea-positivity}.
\end{remark}

\smallskip
\noindent
{\it Acknowledgements.}  We thank Leonardo Mihalcea for useful comments.  This project began in March 2010 at the AIM workshop on Localization Techniques in Equivariant Cohomology, and we thank William Fulton, Rebecca Goldin, and Julianna Tymoczko for organizing that meeting.

\section{Setup}

%
%

We  assume $G$ is an adjoint group, so that the simple roots $\Delta=\{\alpha_1,\ldots,\alpha_n\}$ form a basis for the character group of $T$.  We fix the basis $-\Delta=\{-\alpha_1,\ldots,-\alpha_n\}$ of \emph{negative} simple roots, and use it to identify $T$ with $(\CC^*)^n$.

\subsection{Equivariant cohomology}\label{ss:eq}

Let $\EE T \to \BB T$ be the universal principal $T$-bundle; that is, $\EE T$ is a contractible space with a free right $T$-action, and $\BB T = \EE T / T$.  By definition, the equivariant cohomology of a $T$-variety $Z$ is the ordinary (singular) cohomology of the \emph{Borel mixing space} $\EE T \times^T Z$.  (This notation means quotient by the relation $(e\cdot t, z) \sim (e, t\cdot z)$.)  While $\EE T$ is infinite-dimensional, it may be approximated by finite-dimensional smooth varieties.  We will set $\EE = (\CC^m\setminus\{0\})^n$, with $T=(\CC^*)^n$ acting by scaling each factor.  For fixed $k$ and $m\gg 0$, one has natural isomorphisms
\[
  H_T^*Z := H^*(\EE T\times^T Z) \isom H^*(\EE \times^T Z),
\]
so any given computation may be done with these approximation spaces.

Note that $\BB = \EE/T$ is isomorphic to $(\PP^{m-1})^n$.  For a $T$-variety $Z$, we will generally use calligraphic letters to denote the corresponding approximation space: $\Zz = \EE \times^T Z$, always understanding a suitably large fixed $m$.  This is a fiber bundle over $\BB$, with fiber $Z$.

For each $j=0,\ldots,m-1$, we fix transverse linear subspaces $\PP^{m-1-j}$ and $\tilde\PP^j$ inside $\PP^{m-1}$, and for each multi-index $J=(j_1,\ldots,j_n)$ with $0\leq j_i \leq m-1$, we set
\[
  \BB_J = \tilde\PP^{j_1} \times \cdots \times \tilde\PP^{j_n} \qquad \text{ and } \qquad
  \BB^J = \PP^{m-1-j_1} \times \cdots \times \PP^{m-1-j_n}.
\]
So $\dim\BB_J = \codim\BB^J = |J| := j_1 + \cdots + j_n$.  Similarly, write $\Zz_J = (\pi^T)^{-1}\BB_J$ and $\Zz^J = (\pi^T)^{-1}\BB^J$, where $\pi^T:\Zz\to\BB$ is the projection.  The notation is chosen to suggest an identification of the pushforward for this fiber bundle with the equivariant pushforward $\pi^T_*:H_T^*Z \to H_T^*(\pt)$.

Let $\Oo_i(-1)$ be the tautological bundle on the $i$th factor of $\BB=(\PP^{m-1})^n$.  The choice of basis $-\Delta$ for the character group of $T$ yields an equality $\alpha_i = c_1(\Oo_i(1))$.  If $\alpha=a_1\alpha_1+\cdots+a_n\alpha_n$ is a root, we will sometimes write $\Oo(\alpha) = \Oo_1(a_1)\otimes \cdots \otimes \Oo_n(a_n)$ for the corresponding line bundle, so $c_1(\Oo(\alpha))=\alpha$.  Note that $\Oo(\alpha)$ is globally generated if and only if $\alpha$ is a positive root.

From the definitions, we have
\[
 [\BB^J]=\alpha^J := \alpha_1^{j_1}\cdots \alpha_n^{j_n}
\]
in $H^*\BB$.  As a consequence, suppose $c= \sum_J c_J \alpha^J$ is an element of $H^*\BB = H_T^*(\pt)$, with $c_J\in\ZZ$.  Using Poincar\'e duality on $\BB$, we have $c_J = \pi^\BB_*(c \cdot [\BB_J])$, where $\pi^\BB$ is the map $\BB \to\pt$.

When $c=\pi^T_*(\sigma)$ comes from a class $\sigma \in H_T^*Z = H^*\Zz$ for a complete $T$-variety $Z$, we have
\begin{align}\label{eq:cj}
  c_J 
      &= \pi^\Zz_*(\sigma \cdot [\Zz_J] ),
\end{align}
using the projection formula and the fact that $(\pi^T)^*[\BB_J]=[\Zz_J]$.  (The latter holds since $\pi^T:\Zz\to \BB$ is flat; for a more general argument in the case where $Z$ is Cohen-Macaulay, see \cite[Lemma, p. 108]{fpr}.)

%
\subsection{Stable maps}

We briefly summarize some basic facts about the space of stable maps; proofs and details may be found in \cite{fp}.  As always, $X=G/P$.  The (coarse) moduli space $\Mbar = \Mbar_{0,r+1}(X,\dd)$ parametrizes data $(f,C,p_1,\ldots,p_{r+1})$, where $C$ is a connected nodal curve of genus $0$, and $f:C \to X$ is a map with $f_*[C] = \dd$ in $H_2(X,\ZZ)$.  (Stability means that any irreducible component of $C$ which is collapsed by $f$ has at least three ``special'' points, i.e., marked points $p_i$ or nodes.)

The space of stable maps is an irreducible projective variety of dimension
\[
  \dim \Mbar = \dim X + \langle c_1(TX), \dd \rangle + r-2,
\]
and has quotient singularities, and therefore rational singularities; in particular, it is Cohen-Macaulay.  The locus parametrizing maps with irreducible domain is a dense open subset $M=M_{0,r+1}(X,\dd) \subseteq \Mbar$, and the complement is a divisor $\bdy = \Mbar\setminus M$.

There are natural \emph{evaluation maps} $\ev_i:\Mbar \to X$, defined by sending a stable map $(f,C,p_1,\ldots,p_{r+1})$ to $f(p_i)$.  The group $G$ acts on $\Mbar$ by $g\cdot (f,C,\{p_i\}) = (g\cdot f,C,\{p_i\})$, and the evaluation maps are equivariant for the actions of $G$ on $\Mbar$ and $X$.  Considering the induced action of $T\subset G$, we obtain maps $\ev_i^T: \Mmbar \to \Xx$ on Borel mixing spaces, which commute with the projections to $\BB$.

\begin{remark}
The significance of $\Mbar$ being Cohen-Macaulay is that the usual apparatus of intersection theory applies; see especially Lemma~\ref{l:fp} below.  In fact, the corresponding moduli stack is smooth, so one could argue directly using intersection theory on stacks.
\end{remark}

\section{A group action}

In \cite{anderson} and \cite{agm}, a large group action on the mixing space $\Xx$ was constructed; we describe it here.  The idea is to mix the transitive action of $(PGL_m)^n$ on $\BB$ with a ``fiberwise'' action by Borel groups.  Let $T$ act on $G$ by conjugation, and let $\Gg = \EE \times^T G$ be the corresponding group scheme over $\BB$.  Because $T$ acts by conjugation, the evident action $(\EE \times G) \times_\EE (\EE \times X) \to \EE \times X$ descends to an action $\Gg \times_\BB \Xx \to \Xx$.

Let $U\subset B \subset G$ be the unipotent radical of $B$, and let $\Uu \subset \Bb \subset \Gg$ be the corresponding group bundles over $\BB$.  As a variety, $\Uu$ is isomorphic to the vector bundle $\bigoplus_{\alpha\in R^+} \Oo(\alpha)$ on $\BB$, where the sum is over the positive roots.  Now consider the group of sections $\Gamma_0 = \Hom_\BB(\BB,\Uu)$; this is a connected algebraic group over $\CC$.  As observed in \S\ref{ss:eq}, each $\Oo(\alpha)$ is globally generated.  It follows that for each $x\in \BB$, the map $\Gamma_0 \to \Uu_x$ given by evaluating sections at $x$ is surjective, and therefore we have:

\begin{lemma}[{\cite[Lemma~6.3]{agm}}]
Let $\Gamma$ be the \define{mixing group} $\Gamma_0 \rtimes (PGL_m)^n$, where $(PGL_m)^n$ acts on $\Gamma_0$ via its action on $\BB$.  Then $\Gamma$ is a connected linear algebraic group acting on $\Xx$, with (finitely many) orbits whose closures are the Schubert bundles $\Xx(w)$.
\end{lemma}

Similarly, the group $\Gamma^{(r)} = \Gamma_0^r \rtimes (PGL_m)^n$ acts on the $r$-fold fiber product $\Xx \times_\BB \cdots \times_\BB \Xx$, with orbit closures $\Xx(w_1)\times_\BB \cdots \times_\BB \Xx(w_r)$.

\section{Transverality}\label{s:transverse}

A map $f:Y \to X$ is said to be \define{dimensionally transverse} to a subvariety $W \subseteq X$ if $\codim_Y(f^{-1}W) = \codim_X(W)$.  We will need the following version of Kleiman's transversality theorem; see \cite{kleiman} and \cite{speiser}.  As a matter of notation, if a group $\Gamma$ acts on $X$, we write $\gamma f: \gamma Y \to X$ for the composition $Y\xrightarrow{f} X \xrightarrow{\cdot\gamma} X$, i.e., the translation of $f$ by the action of $\gamma\in\Gamma$.

\begin{proposition}\label{p:transverse}
Let $\Gamma$ be a group acting on a smooth variety $X$, and suppose $f:Y \to X$ is dimensionally transverse to the orbits of $\Gamma$.  Assume $Y$ is Cohen-Macaulay.  Let $g: Z \to X$ be any map.  Then for a general element $\gamma\in\Gamma$, the fiber product $V_\gamma = \gamma Y \times_X Z$ has dimension equal to $\dim Y + \dim Z - \dim X$.
\end{proposition}

The essential point in the proof is that the hypotheses imply the map $\Gamma \times Y \to X$ is flat.

We will also use the following lemma:
\begin{lemma}[{\cite[Lemma, p.~108]{fpr}}]\label{l:fp}
Let $f:Z \to X$ be a morphism from a pure-dimensional Cohen-Macaulay scheme $Z$ to a nonsingular variety $X$, and let $W\subseteq X$ be a closed Cohen-Macaulay subscheme of pure codimension $d$.  Let $V = f^{-1}W$, and assume $\codim_Z(V)=d$.  Then $V$ is Cohen-Macaulay, and $f^*[W]=[V]$.
\end{lemma}

Now resume the previous notation, so $X=G/P$ and $\Mbar=\Mbar_{0,r+1}(X,\dd)$.  Since each evaluation map $\ev_{i}:\Mbar \to X$ is $G$-equivariant, it is flat.  If $W\subseteq X$ is any Cohen-Macaulay subscheme of codimension $d$, it follows that $\ev_i^{-1}W\subseteq\Mbar$ has the same properties, and similarly, $(\ev_i^T)^{-1}\Ww \subseteq \Mmbar$.  In particular, the subscheme
\[
  \Zz = (\ev_{r+1}^T)^{-1}(\Xx(w)) \subseteq \Mmbar
\]
is Cohen-Macaulay of codimension $\dim X-\ell(w)$, and we have $[\Zz]=(\ev_{r+1}^T)^*(x(w))$ by Lemma~\ref{l:fp}.  Similarly, we have
\begin{align}\label{eq:zj}
  [\Zz_J] = (\ev_{r+1}^T)^*(x(w))\cdot[\Mmbar_J]
\end{align}

Consider the map $\ev=\ev_1\times\cdots\times\ev_r:\Mbar \to X^{r}$ and the corresponding map on mixing spaces $\ev^T: \Mmbar \to \Xx^{r}$.  Let $\Yy = \Yy(v_1)\times_\BB \times \cdots \times_\BB \Yy(v_r)$, and let $f$ be the inclusion of $\Yy$ in the $r$-fold fiber product $\Xx^{r}$. 

\begin{lemma}\label{l:intersect}
Let $\gamma = (\gamma_1,\ldots,\gamma_r)$ be a general element in $\Gamma^{(r)}$.

\begin{enumerate}
\item The intersection
\begin{align*}
  V_\gamma &= (\ev^T_1)^{-1}(\gamma_1\Yy(v_1)) \cap \cdots \cap (\ev^T_r)^{-1}(\gamma_r\Yy(v_r)) \cap \Zz_J \\
   &= \gamma\Yy \times_{\Xx^r} \Zz_J
\end{align*}
is Cohen-Macaulay and pure-dimensional, of dimension $\dim\Mbar+|J|-\dim X+\ell(w)-\ell(v_1)-\cdots-\ell(v_r)$.  (In the fiber product, $\Zz_J$ maps to $\Xx^r$ by the restriction of $\ev^T$.)\label{inta}

\medskip

\item Similarly, the intersection
\begin{align*}
  \partial V_\gamma &= (\ev^T_1)^{-1}(\gamma_1\Yy(v_1)) \cap \cdots \cap (\ev^T_r)^{-1}(\gamma_r\Yy(v_r)) \cap \Zz_J \cap \bbdy \\
   &= \gamma\Yy \times_{\Xx^r}(\Zz_J \cap \bbdy)
\end{align*}
has pure dimension $\dim\Mbar+|J|-\dim X+\ell(w)-\ell(v_1)-\cdots-\ell(v_r)-1$. \label{intb}
\end{enumerate}

In particular, when $\dim\Mbar+|J|-\dim X+\ell(w)-\ell(v_1)-\cdots-\ell(v_r)=0$, the intersection $V_\gamma$ consists of finitely many points contained in $\Mm$.
\end{lemma}

\begin{proof}
Note that $\Zz_J$ is Cohen-Macaulay (since $\Zz$ is), of dimension $\dim\Mbar + |J| - \dim X + \ell(w)$.  Each opposite Schubert bundle $\Yy(v)$ intersects each $\Gamma$-orbit closure $\Xx(w)$ properly, so the map $f:\Yy \hookrightarrow \Xx^r$ is dimensionally transverse to the $\Gamma^{(r)}$-orbits.  The first statement follows by an application of Proposition~\ref{p:transverse}.

The second statement is proved similarly; note that the divisor $\bdy$ is Cohen-Macaulay and $G$-invariant, and the same argument as before shows that $\Zz_J \cap \bbdy$ is a Cohen-Macaulay divisor in $\Zz_J$.
\end{proof}

We can now prove Theorem~\ref{t:main}.  In fact, it follows immediately from \eqref{eq:cj}, together with a more precise statement.

\begin{theorem}\label{t:count}
Write $I^T_\dd(y(v_1)\cdots y(v_r)\cdot x(w)) = \sum c_J \alpha^J$ in $H_T^*(\pt)$.  Then, with notation as in Lemma~\ref{l:intersect}, we have
\[
  c_J = \deg( V_\gamma )
\]
when $\dim\Mbar+|J|-\dim X+\ell(w)-\ell(v_1)-\cdots-\ell(v_r)=0$, and $c_J=0$ otherwise.

In particular, since $V_\gamma$ is an effective cycle, $c_J$ is a nonnegative integer.
\end{theorem}

\begin{proof}
Using \eqref{eq:cj} from \S\ref{ss:eq}, we have
\[
  c_J = \pi^{\Mmbar}_*((\ev_1^T)^*y(v_1)\cdots (\ev_r^T)^*y(v_r) \cdot (\ev_{r+1}^T)^*x(w) \cdot [\Mmbar_J]).
\]
The claim is that $(\ev_1^T)^*y(v_1)\cdots (\ev_r^T)^*y(v_r) \cdot (\ev_{r+1}^T)^*x(w) \cdot [\Mmbar_J] = [V_\gamma]$ in $H^*\Mmbar$.

First observe that $(\ev_1^T)^*y(v_1)\cdots (\ev_r^T)^*y(v_r) = (\ev^T)^*(y(v_1)\times \cdots \times y(v_r))$.  Since $\Gamma^{(r)}$ is connected, we have $[\gamma\Yy] = [\Yy] = y(v_1)\times \cdots \times y(r)$ in $H^*(\Xx^r) = H_T^*(X^r)$.  By the same argument as in the paragraph after Lemma~\ref{l:fp}, we have $[(\ev^T)^{-1}(\gamma\Yy)] = (\ev^T)^*(y(v_1)\times \cdots \times y(v_r))$.

By \eqref{eq:zj}, we have $[\Zz_J] = (\ev^T_{r+1})^*x(w)\cdot[\Mmbar_J]$.  Since $(\ev^T)^{-1}(\gamma\Yy)$ and $\Zz_J$ intersect properly in $V_\gamma$ by Lemma~\ref{l:intersect}, we have $[(\ev^T)^{-1}(\gamma\Yy)]\cdot [\Zz_J] = [V_\gamma]$, as desired.
\end{proof}

\begin{remark}\label{r:genx}
Let $\Mbar_{0,r+1}$ be the moduli space of stable curves with $r+1$ marked points; this is a nonsingular projective variety of dimension $r-2$.  Since $T$ acts trivially on this space, the corresponding mixing space is $\Mmbar_{0,r+1} = \BB\times\Mbar_{0,r+1}$.  The forgetful map $\phi:\Mbar \to \Mbar_{0,r+1}$ induces a map $\Mmbar \to \Mmbar_{0,r+1}$.  Let $\tilde\phi:\Mmbar \to \Mbar_{0,r+1}$ be the composition with the second projection, and for $x\in \Mbar_{0,r+1}$, write $\Mmbar(x)=\tilde\phi^{-1}(x)$.  Using the notation of Lemma~\ref{l:intersect}, the same arguments used in the proof of the lemma also establish the following dimension counts:
\begin{enumerate}
\item Let $V_\gamma(x) = V_\gamma\cap\Mmbar(x)$.  Then $V_\gamma(x)$ is Cohen-Macaulay, of pure dimension $\dim\Mbar+|J|-(\dim X-\ell(w))-\ell(v_1)-\cdots-\ell(v_r)-(r-2)$.

\medskip

\item Let $\partial V_\gamma(x) = \partial V_\gamma\cap\Mmbar(x)$.  Then $\partial V_\gamma(x)$ is Cohen-Macaulay, of pure dimension $\dim\Mbar+|J|-(\dim X-\ell(w))-\ell(v_1)-\cdots-\ell(v_r)-(r-2)-1$.
\end{enumerate}
\end{remark}



\end{document}